\documentclass[11pt, oneside]{amsart}
\pdfoutput=1
\usepackage[utf8]{inputenc}
\usepackage[T1]{fontenc}
\usepackage[pdftex, paperwidth=210mm, paperheight=11in]{geometry}
\usepackage[unicode,hidelinks]{hyperref}
\usepackage{xcolor}
\hypersetup{
    colorlinks,
    linkcolor={red!20!black},
    citecolor={blue!30!black},
    urlcolor={blue!40!black}
    }

\usepackage{mathptmx}
\usepackage{newtxmath}
\usepackage{mathbbol}
\usepackage{microtype}

\usepackage{enumitem}
\usepackage{hyphenat}
\usepackage{xspace}
	\xspaceaddexceptions{\, \: \nb \nobreakdash \text \textup } 

\makeatletter

\swapnumbers

\theoremstyle{plain}
    \newtheorem{theorem}{Theorem}[section]
    \newtheorem{lemma}[theorem]{Lemma}
    \newtheorem{proposition}[theorem]{Proposition}

\theoremstyle{definition}
    \newtheorem{definition}[theorem]{Definition}
    \newtheorem*{theorem*}{Theorem}
    \newtheorem*{notation*}{Notation}

\theoremstyle{remark}
    
    \newtheorem*{remark*}{Remark}
    
\usepackage{needspace}

\newcommand{\comp}{{\sim}}
\newcommand{\restr}{{\restriction}}

\setlist[enumerate]{leftmargin=20pt, topsep=2pt}
\setlist[itemize]{leftmargin=20pt}
\setenumerate[1]{label=\upshape{(\arabic*)}}

\newcommand{\nb}{\nobreakdash}

\let\@O\O
\renewcommand{\O}{\ifmmode\varnothing\else\@O\fi}

\newcommand{\subs}{\subseteq}
\newcommand{\sups}{\supseteq}

\newcommand{\X}{\times}

\newcommand{\MedBigcup}{\textstyle{\bigcup}}
\newcommand{\MedBigcap}{\textstyle{\bigcap}}

\newcommand{\WO}[1]{\ensuremath{\mathrm{WO}_{\smash{#1}}}\xspace}
\newcommand{\LO}{\ensuremath{\mathrm{LO}}\xspace}
\newcommand{\wCK}{{\omega_{1}^{\scriptscriptstyle\mathsf{CK}}}}

\newcommand{\Set}[2]{\ensuremath{\{\mkern1mu{#1}\mid{#2}\mkern1mu\}}\xspace}
\newcommand{\Angles}[1]{\ensuremath{\langle{#1}\rangle}\xspace}

\newcommand{\Equiv}{\;\Leftrightarrow\;}
\newcommand{\LEquiv}{\;\Leftarrow\mkern-12mu\Rightarrow\;}
\newcommand{\Implies}{\;\Rightarrow\;}

\newcommand{\SigmaOne}{\ensuremath{\Sigma_{1}^{1}}\xspace}
\newcommand{\PiOne}{\ensuremath{\Pi_{1}^{1}}\xspace}
\newcommand{\DeltaOne}{\ensuremath{\Delta_{1}^{1}}\xspace}

\newcommand{\cB}{\ensuremath{\mathcal{B}}\xspace}
\newcommand{\cF}{\ensuremath{\mathcal{F}}\xspace}
\newcommand{\cG}{\ensuremath{\mathcal{G}}\xspace}
\newcommand{\cI}{\ensuremath{\mathcal{I}}\xspace}
\newcommand{\cN}{\ensuremath{\mathcal{N}}\xspace}
\newcommand{\cO}{\ensuremath{\mathcal{O}}\xspace}
\newcommand{\cP}{\ensuremath{\mathcal{P}}\xspace}
\newcommand{\cS}{\ensuremath{\mathcal{S}}\xspace}
\newcommand{\cT}{\ensuremath{\mathcal{T}}\xspace}
\newcommand{\cX}{\ensuremath{\mathcal{X}}\xspace}

\newcommand{\Dtopology}{$\Delta$\nobreakdash-topol\-o\-gy\xspace}
\newcommand{\SDtopology}{${\Sigma\Delta}$\nobreakdash-topol\-o\-gy\xspace}
\newcommand{\GHtopology}{\ensuremath{\mathcal{S}_{\mathrm{GH}}}}

\newcommand{\Sdelta}{\ensuremath{\mathcal{S}_\Delta}\xspace}

\newcommand{\thinoverline}[1]{\fontdimen8\textfont3=.4pt\overline{#1}}
\newcommand{\SdeltaCl}[1]{\ensuremath{\thinoverline{#1}{\strut^\Delta}}}

\renewcommand{\le}{\leqslant}

\makeatother

\begin{document}
\title[A note on an effective Polish topology]
						{A note on an effective Polish topology \\and Silver's Dichotomy theorem}
\author
						{Ramez L. Sami}
\address
						{Department of Mathematics. Université Paris\nobreakdash-Diderot. 75205 Paris Cedex 13, France.}
\email
						{sami@univ-paris-diderot.fr}
\subjclass[2010]
						{Primary: 03E15; Secondary: 28A05, 54H05}

\begin{abstract}
\ \par
    \begin{itemize}[label={$\scriptstyle\bullet$}, leftmargin=9pt, topsep=0pt, itemsep=0pt]
    
    \item We define a Polish topology inspired from the Gandy-Harrington topology and show how it can be used to prove Silver's dichotomy theorem while remaining in the Polish realm.
    
    \item In this topology, a $\Pi^1_1$ equivalence relation decomposes into a ``sum'' of a clopen relation and a meager one.
    
    \item We characterize it as the largest regular toplogy with a basis included in $\Sigma^1_1$.
\end{itemize}

\end{abstract}

\vspace*{-4\baselineskip}
\begin{flushright}
    	[\textit{To appear in}: Proc. Amer. Math. Soc. \textbf{147} (2019),\,no.\,9]
	\medskip{}
\end{flushright}

\pdfbookmark[1]{A note on an effective Polish topology and~Silver's Dichotomy theorem}{}
\maketitle

Jack Silver's remarkable dichotomy theorem is the statement: \emph{A $\boldsymbol{\Pi}^{1}_{1}$ equivalence relation on a Polish space either has countably many classes, or a perfect set of mutually inequivalent elements}. 

Silver's proof \cite{Silver} mobilized $\aleph_{1}$ distinct cardinals, and forced over a model of set theory. \nohyphens{Mercifully}, some years later, Leo Harrington produced a much milder forcing proof \cite{Harrington}, formalizable in analysis. Alain Louveau evolved this into a purely topological proof \cite{Louveau}, bringing forth the Gandy-Harrington topology: $\GHtopology =$ \emph{the topology generated by the \textup(lightface\textup) \SigmaOne subsets of the Baire space \cN}.

While $(\cN, \GHtopology)$ is a Baire space in a strong sense,\footnote{\,It was shown by Louveau to be a \emph{strong Choquet space}, see \cite[§4.2]{HKL}.} it is not Polish, disallowing the use of familiar ``Polish space arguments''. See \cite[§9]{Kechris-Martin}, \cite[§5.3]{Gao}, or \cite[§8]{Mansfield-Weitkamp} for detailed expositions.
\vspace{2pt}

The purpose of this brief note is to introduce the \SDtopology, inspired from \GHtopology, which is Polish, and provides a streamlined proof of Silver's theorem — still, the main ideas go back to Harrington and Louveau. This topology turns out to be a natural object, being characterizable as the largest Polish (or just regular) \emph{effective topology} on \cN — meaning, with a basis included in \SigmaOne.\:\:

In the same spirit, we believe that the proofs of several important results built upon the Gandy-Harrington topology (such as, \emph{inter alia}, the main results of Harrington, Kechris \& Louveau \cite{HKL}) can be simplified and streamlined in the more natural context of the \SDtopology.%
	\footnote{\,The present note greatly antedates the more recent ``back-to-classical'' movement developed \emph{con maestria} in Ben~Miller's~\cite{Miller}. We still hold that effective methods will often yield simpler proofs of stronger and finer results.}
\smallskip{}

In §\ref{sec:The-topology} we define the topology, prove it is Polish, and develop its basic properties. We apply these to give in §\ref{sec:Pi-equivalence-relations} the proof of Silver's theorem, showing on the way that a \PiOne equivalence decomposes into a ``sum'' of a clopen relation and a meager one. In §\ref{sec:Naturality} the characterization theorem is proved.

\section{\texorpdfstring{The \SDtopology}{The Σ∆-topology}}
\label{sec:The-topology}

\cN shall denote the standard Baire space $\omega^\omega = \mathbb{N}^{\mathbb{N}}$ (the set of \emph{reals}). Throughout, \SigmaOne, \PiOne, and \DeltaOne denote the \emph{lightface} classes of subsets of \cN, or $\cN^{2}$. We refer to Moschovakis \cite{Moschovakis} for effective descriptive set theory: all results in this note translate readily to recursively presented Polish spaces.

\begin{definition}
\ \par
\begin{enumerate}
	\item The \Dtopology is the topology with basis the \DeltaOne subsets of \cN, denoted \Sdelta.
	\item The \SDtopology is the topology with basis the \SigmaOne subsets of \cN that are \Sdelta\nb-closed.\\
	Observe that a set $A \subs \cN$ is \Sdelta\nb-closed just in case $A = \bigcap \Set{D \in \DeltaOne}{D \sups A}$.
\end{enumerate}
\end{definition}

The \SDtopology is the finer one, and its basis, as defined above, consists of clopen sets — both topologies are \emph{zero-dimensional}. Dually, a \PiOne set which is a union of \DeltaOne sets is also clopen for the \SDtopology. Observe that \DeltaOne functions $\cN \to \cN$ are continuous for either one.
\smallskip

We are only interested here in the \SDtopology. The \Dtopology serves only as a prop, though it may be of use in other contexts. We first show that both topologies are Polish — indeed, under adequate closure conditions, a topology generated by a countable collection of Borel sets, over a base Polish topology, is Polish. For ${\cX \subs \cP(U)}$, let $\Angles{\cX}_U$, or just \Angles{\cX}, denote the topology generated by \cX. The following proposition is quite familiar.

\begin{lemma}
\ \par
\begin{enumerate}
	\item Let $\cT_{i}$, $i < \omega$, be Polish topologies on a set $U$ such that $\bigcap_{i < \omega} \cT_{i}$ is Hausdorff. The topology \Angles{\bigcup_{i < \omega} \cT_{i}} is Polish.
	\item Let \cF be a countable collection of closed sets in a Polish space $(U, \cT)$. \Angles{\cT \cup \cF} is a Polish topology.
\end{enumerate}
\end{lemma}

\begin{proof}
1.\enskip{}Set $\cP = \prod_{i} ( U, \cT_{i} )$. Being a countable product of Polish spaces, \cP is Polish. 
 
Set $\cT_\omega = \Angles{\bigcup_{i} \cT_{i}}$. Easily, $(U, \cT_\omega)$ is homeomorphic to the diagonal $D$ of \cP, and the hypothesis ``\,$\bigcap_{i} \cT_{i}$ is Hausdorff\,'' entails that $D$ is closed in \cP, thus $D$ is a Polish subspace. Hence the conclusion.\smallskip{}

\noindent 2.\enskip{}For a closed $F \subs U$, set $O = \comp F$. Both $F$ and $O$ are Polish subspaces. The space $(U, \Angles{\cT \cup \{F\}})$ is homeomorphic to the direct sum $F \oplus O$, hence it is Polish. Now, for a countable collection of closed sets \cF, \Angles{\cT \cup \cF} is generated by $\bigcup_{F \in \cF}\Angles{\cT \cup \{F\}}$. Thus the conclusion, invoking (1).
\end{proof}

\begin{theorem}
The \Dtopology and the \SDtopology are both Polish.%
	\footnote{\,Dominique Lecomte has pointed to us that Louveau was first to show that the $\Delta$\nobreakdash-topology is Polish (with a fairly elaborate argument).}
\end{theorem}

\begin{proof}
For $\nu$, $1 \le \nu \le \wCK$, set $\cT_\nu = \Angles{\cB_\nu}$, where
\[
    	\cB_\nu = \MedBigcup_{\xi < \nu}(\Sigma^{0}_\xi \cup \Pi^{0}_\xi).
\]
We check, by induction on $\nu$, that the topology $\cT_{\nu}$ is Polish.

$\cT_{1}$ is the standard topology on \cN. For $\nu = \mu+1$, observe that $\cT_{\mu+1} = \Angles{\cT_\mu \cup \Sigma^{0}_\mu \cup \Pi^{0}_\mu}$, and $\Sigma^{0}_\mu \subs \cT_\mu$, hence $\cT_{\mu+1} = \Angles{\cT_{\mu} \cup \Pi^{0}_\mu}$ and $\Pi^{0}_\mu$ is a countable collection of $\cT_\mu$\nb-closed sets. Thus, by the lemma-(2), $\cT_{\mu+1}$ is Polish. Finally, for $\nu$ limit, easily $\cT_\nu = \Angles{\bigcup_{1 \le \xi < \nu}\cT_\xi}$, hence $\cT_\nu$ is Polish, by the lemma-(1).

Now $\DeltaOne = \cB_\wCK$, thus the \Dtopology $\Sdelta = \Angles{\cB_\wCK}$ is Polish.

As to the \SDtopology: it is generated over \Sdelta by a countable collection of closed sets (the~\Sdelta\nb-closed \SigmaOne sets), it is thus Polish, by the lemma-(2).
\end{proof}

\begin{notation*}
    \cS shall denote the \SDtopology on \cN, and $\cS_{2}$ the \SDtopology on $\cN^{2}$. $\cS_{2}$ is just a recursively homeomorphic copy of \cS. \;$\cS \X \cS$ denotes the usual product topology.
\end{notation*}

$\cS_{2}$ is strictly finer than $\cS \X \cS$, as the diagonal \Set{(x,x)}{x \in \cN} is open for $\cS_{2}$, and not for $\cS \X \cS$. Note that \SigmaOne sets being \emph{analytic} for \cS, $\cS \X \cS$ or $\cS_{2}$, as may be, have the Baire property there.
\medskip{}

\noindent\textbullet\enskip{}\emph{Save for explicit mention, topological terms henceforth are relative to the \SDtopology \cS, or to the product topology $\cS \X \cS$}.
\medskip{}

Let \LO and \WO{} denote the sets of reals coding linear and well-orders with field in $\omega$. For~$\alpha < \omega_{1}$, \WO{\alpha} and \WO{<\alpha} are self-explanatory. \LO is clopen. If $u \in \LO$, then for $k \in \omega$, $u \restr k$ codes $<_{u}$~restricted to \Set{n \in \omega}{n <_{u} k}, else $u \restr k$ codes $\O$. The function $(u,k) \mapsto u \restr k$ is continuous.

\begin{theorem}
\label{prop:Sigma-is-nonmeager}
\ \par
\begin{enumerate}
	\item The set $\cG = \Set{x \in \cN}{\omega_{1}^{x} = \wCK}$ is a dense $G_\delta$.
	\item Every nonempty \SigmaOne subset of \cN is nonmeager \textup(and identically for subsets of $\cN^{2}$ relative to~$\cS_{2}$\textup).
\end{enumerate}
\end{theorem}

\begin{proof}
\noindent 1.\enskip{}By Gandy's basis theorem, every nonempty \SigmaOne set meets \cG. Basic open sets being \SigmaOne, \cG~is dense. Note that, for $\xi < \wCK$, \WO{\xi} is \DeltaOne thus:

\begin{itemize}[leftmargin=16pt, itemsep=3pt, label=\textbf{--}]
 \item \WO{<\wCK} is clopen, for it is a \PiOne union of \DeltaOne sets, and

\item \WO{\wCK} is closed, as it is an intersection of clopen sets:
\[
    	x \in \WO{\wCK} \LEquiv x \in \LO \And x \notin \WO{<\wCK} \And \forall k \in \omega(x \restr k \in \WO{<\wCK}).
\]
\end{itemize}

\noindent{}Let now $(f_{k})_{k < \omega}$ list the total recursive functions $\cN \to \cN$. The $f_{k}$'s are continuous, and 
\[
    	\cG = \Set{x \in \cN}{\forall k (f_{k}(x) \notin \WO{\wCK})} = \MedBigcap_{k} f_{k}^{-1}[\comp\WO{\wCK}].
\]
Hence \cG is a countable intersection of open sets: it is a $G_\delta$, and dense.
\smallskip{}

\noindent 2.\enskip{}Let $A$ be \SigmaOne nonempty. Set $A^{-} = A \cap \cG$, by Gandy's theorem $A^{-} \neq \O$. $A^{-}$ being \SigmaOne, let $F:\cN \to \cN$ be recursive, such that $x \in A^{-} \Equiv F(x) \notin WO$. Note that
\[
    	x \in A^{-} \LEquiv x \in \cG \And F(x) \notin \WO{<\wCK}\,,
\]
thus $A^{-}$ is a nonempty intersection of the dense $G_\delta$ set \cG with a clopen set: it is nonmeager.
\end{proof}

Observe that, whereas in the Gandy-Harrington topology \SigmaOne sets are open (and nonmeager), here nonempty \SigmaOne sets are merely nonmeager, which is strength enough for the applications. No Polish topology on \cN can include the class \SigmaOne (see the comments following \ref{def:Effective-Topology}).

\begin{proposition}
\label{prop:If-A-T-comeager}
If $A \subs \cN$ is comeager, $A^{2}$ is $\cS_{2}$\nb-comeager.

Similarly, if $A$ is comeager in an open set $U \subs \cN$, $A^{2}$ is $\cS_{2}$\nb-comeager in $U^{2}$.
\end{proposition}

\begin{proof}
We may take $A$ to be a dense $G_\delta$. Easily, $A \X \cN$ is $\cS_{2}$\nb-$G_\delta$.
We now check that it is also $\cS_{2}$\nb-dense, hence $\cS_{2}$\nb-comeager.

Let $O \subs \cN^{2}$ be $\cS_{2}$\nb-basic\nb-open, nonempty. Its first projection $\pi_{1}[O]$ is \SigmaOne nonempty, hence nonmeager by Theorem\,\ref{prop:Sigma-is-nonmeager}\nb-(2). Consequently, $\pi_{1}[O] \cap A \neq \O$, i.e., $O \cap (A \X \cN) \neq \O$.

Symmetrically, $\cN \X A$ is $\cS_{2}$\nb-comeager. Hence $A^{2} = (A \X \cN) \cap (\cN \X A)$ is $\cS_{2}$\nb-comeager.
\end{proof}

\section{\texorpdfstring{\PiOne equivalence relations}{Π₁¹  equivalence relations}}
\label{sec:Pi-equivalence-relations}

The following result may well have started the descriptive set theory of equivalence relations.

\begin{theorem*}[Jan Mycielski]
Let $E$ be an equivalence relation on a Polish space $Z$. If $E$ is meager in~the product $Z \X Z$, then $E$ has a perfect set of mutually inequivalent elements.
\end{theorem*}

\begin{proof}
See \cite[§5.3.1]{Gao}, or \cite[§9.2]{Kechris-Martin} for a more general version of the theorem.
\end{proof}

Recall that, by the Kuratowski-Ulam theorem, an equivalence relation on a Polish space, having the Baire property, is meager in the product space if, and only if, all of its classes are meager.

\begin{lemma}
Let $E$ be a \PiOne equivalence relation on \cN, $C$ a nonmeager $E$\nb-class, and let $C^{\circ}$ be its~interior.
\begin{enumerate}
	\item $C^{\circ} \neq \O$.
	\item $C^{\circ}$ is a union of \DeltaOne sets.
\end{enumerate}
\end{lemma}

\begin{proof}
1.\enskip{}Let $U$ be basic\nb-open nonempty such that $C$ is comeager in $U$. We claim that $U \subs C$. Using Proposition \ref{prop:If-A-T-comeager}, $C^{2}$ is $\cS_{2}$\nb-comeager in the open set $U^{2}$, and evidently $C^{2} \cap (U^{2} - E) = \O$. Thus, the \SigmaOne set $U^{2} - E$ is $\cS_{2}$-meager. By Theorem\,\ref{prop:Sigma-is-nonmeager}\nb-(2), it must be empty, that is $U^{2} \subs E$, and thus $U \subs C$.\looseness=-1

\smallskip{}\noindent 2.\enskip{}Let $U \subs C$ be basic\nb-open, nonempty. $C$ is \PiOne, for $U$ is \SigmaOne and
\[
    	x \in C \LEquiv \forall y \,(y \in U \Implies y\,E\,x).
\]
\SigmaOne separation yields a \DeltaOne set $D$ such that $U \subs D \subs C$. Since $D$ is open, $D \subs C^{\circ}$. Consequently, $C^{\circ} = \bigcup \Set{D \in \DeltaOne}{D \subs C}$.
\end{proof}

\begin{remark*}
    The argument in (1) extends to: \emph{if $C$ is nonmeager in an open $U \subs \cN$, then $(C \cap U)^{\circ} \neq \O$}.
\end{remark*}

\begin{theorem}[Silver's Dichotomy theorem]
    Let $E$ be a \PiOne equivalence relation on \cN. Either $\cN/E$ is countable or $E$ has a perfect set of mutually inequivalent elements.
\end{theorem}

\begin{proof}
We decompose \cN into two clopen subspaces $\cN = H \oplus Z$ such that $E_H$ has countably many classes, and $E_Z$ is meager, where $E_X$ denotes $E \cap (X \X X)$.

Set $H = \bigcup \Set{C^{\circ}}{C \in \cN/E}$, and $Z = \comp H$. 

We check that $H$ is clopen: by the above lemma, it is a union of \DeltaOne sets. It suffices thus to verify that it is \PiOne, indeed,
\[
    x \in H \LEquiv \exists D \in \DeltaOne(x \in D \And D^{2} \subs E).
\]
This yields a \PiOne definition of $H$, through the usual coding of \DeltaOne sets.

\noindent\;\textbf{--}\enskip{}$E_H$ \emph{has countably many classes}, evidently.

\noindent\;\textbf{--}\enskip{}$E_Z$ \emph{is meager}. Indeed, in the clopen subspace $Z$ (if nonempty) all $E_Z$\nb-classes have empty interior, hence are meager, by the last remark. Thus the conclusion, by the Kuratowski-Ulam theorem.

Now if $Z \neq \O$, then by Mycielski's theorem applied to $E_Z$ there is a perfect set $P \subs Z$ of mutually inequivalent reals. $P$ contains a copy of the Cantor space $2^\omega$, which is perfect for the standard topology.
\end{proof}

It is interesting to rephrase the crux of the above argument as a decomposition theorem.

\begin{theorem}
    Given a \PiOne equivalence relation $E$, \cN decomposes into two clopen subspaces $\cN = H \oplus Z$ such that $E_H$ is clopen in $H$, and $E_Z$ is meager in $Z$.
\end{theorem}

\begin{proof}
With $H$ and $Z$ defined as above, it remains to check that $E_H$ is clopen in $H$. This is immediate for, setting $\cI = \Set{C^{\circ}}{C \in \cN/E}$, one has 
\vspace*{-2pt}\jot=2pt 
\begin{align*}
		    	E_H 		\; &= \; \MedBigcup \Set{X \X X}{X \in \cI},\\
		H^{2} - E_H 	\; &= \; \MedBigcup \Set{X \X Y}{X,Y \in \cI \And X \neq Y}.
\end{align*}

\vspace*{-4pt}
\noindent{}Both $E_H$ and its complement in $H^{2}$ are unions of open sets.
\end{proof}

As has been observed, Harrington's argument yields fine effective consequences beyond Silver's dichotomy result, e.g., \emph{in a \PiOne equivalence relation with countably many classes all equivalence classes are \PiOne} — and clopen for the \SDtopology \cS.

\section{\texorpdfstring{Characterizing the \SDtopology}{Characterizing the Σ∆-topology}}
\label{sec:Naturality}

We proceed now to show that the \SDtopology \cS, far from being an \emph{ad-hoc} construction, is the largest Polish topology with a basis included in \SigmaOne — indeed, the largest such regular topology. We~propose here a fairly relaxed notion of effective topology.

\begin{definition}
\label{def:Effective-Topology}
    A topology on \cN is said to be an effective topology if it has a basis included in \SigmaOne.
\end{definition}

Examples of natural effective topologies abound. In this paper, apart from the standard topology, the Gandy-Harrington topology, the \Dtopology and the \SDtopology are all effective. On the other hand, we know that if $\cT \sups \cO$ is a Polish topology on \cN, where \cO is the standard topology, the identity map $(\cN, \cT) \to (\cN, \cO)$ being continuous, all \cT\nb-open sets are Borel for \cO. Consequently, no Polish topology on \cN can include all the \SigmaOne sets.
\smallskip{}

For $A \subs \cN$, \SdeltaCl{A} shall denote the \Sdelta\nb-closure of $A$.

\begin{lemma}
    For all $A \in \SigmaOne$, \SdeltaCl{A} is \SigmaOne, hence \SdeltaCl{A} is \cS\nb-open \emph(as well as \cS\nb-closed, evidently\emph).
\end{lemma}

\begin{proof}
Let $A \in \SigmaOne$. $x \in \SdeltaCl{A}$ if, and only if, every \DeltaOne set containing $x$ meets $A$, i.e., 
\[
    	x \in \SdeltaCl{A} \LEquiv \forall D \in \DeltaOne(x \in D \Implies \exists y \in D \cap A).
\]
Using the usual coding of \DeltaOne sets, \SdeltaCl{A} is seen to be \SigmaOne. By definition of the \SDtopology, $\SdeltaCl{A} \in \cS$.
\end{proof}

\begin{theorem}
    	If \cT is a regular effective topology on \cN, then $\cT \subs \cS$.
\end{theorem}

\begin{proof}
Let $\cB \subs \SigmaOne$ be a basis for \cT. Fix $O \in \cT$, and set $F = \comp O$.

For $x \in O$, by regularity of the topology \cT there are $U_{x} \in \cB$, and $V \in \cT$ such that $x \in U_{x}$, $V \sups F$, and $U_{x} \cap V = \O$. Say $V = \bigcup_{i} V_{i}$, where $V_{i} \in \cB$.

For every $i < \omega$, since $U_{x}$ and $V_{i}$ are \SigmaOne and disjoint, \SigmaOne separation yields $D_{i} \in \DeltaOne$ such that $D_{i} \sups U_{x}$ and $D_{i} \cap V_{i} = \O$. Set $C_{U_{x}} = \bigcap_{i} D_{i}$, $C_{U_{x}}$ is \Sdelta\nb-closed, and $x \in U_{x} \subs C_{U_{x}} \subs O$. Hence $x \in \SdeltaCl{U_{x}} \subs O$. Since $U_{x}$ is \SigmaOne then, by the previous lemma, \SdeltaCl{U_{x}} is \cS\nb-open. $x \in O$ being arbitrary, $O$ is \cS\nb-open. We have shown $\cT \subs \cS$.
\end{proof}

One natural question comes to mind in relation to the \SDtopology: \emph{are there interesting metrics compatible with \cS}?


\bibliographystyle{amsalpha}

\end{document}